\documentclass[11pt,a4paper]{amsart}
\usepackage{amssymb,amsmath}

\usepackage{times}

\usepackage[bookmarks,
colorlinks,
]{hyperref}

\numberwithin{equation}{section}
\newtheorem{theorem}{Theorem}[section]

\newtheorem{lemma}[theorem]{Lemma}

\theoremstyle{definition}

\newtheorem{remark}[theorem]{Remark}

\newtheorem{problem}[theorem]{Problem}

\newcommand\Supp{\operatorname{Supp}}

\newcommand{\Ass}[0]{\operatorname{Ass}}

\newcommand\Tor{\operatorname{Tor}}
\newcommand\Hom{\operatorname{Hom}}

\newcommand\Ext{\operatorname{Ext}}

\newcommand\Spec{\operatorname{Spec}}

\title{A note on the Matlis dual of a certain  injective hull}

\author{Peter Schenzel}

\address{Martin-Luther-Universit\"at Halle-Wittenberg,
Institut f\"ur Informatik, D --- 06 099 Halle (Saale),
Germany}

\email{peter.schenzel@informatik.uni-halle.de}



\begin{document}

\begin{abstract} Let $(R,\mathfrak{m})$ denote a  local ring with $E = E_R(R/\mathfrak{m})$ the injective hull of 
the residue field. Let $\mathfrak{p} \in \Spec R$
denote a prime ideal with $\dim R/\mathfrak{p} = 1$, and let $E_R(R/\mathfrak{p})$ be the injective hull of $R/\mathfrak{p}$.  
As the main result we prove that the Matlis dual $\Hom_R(E_R(R/\mathfrak{p}), E)$ is isomorphic to $\widehat{R_{\mathfrak{p}}}$, 
the completion of $R_{\mathfrak{p}}$, if and only if $R/\mathfrak{p}$ is complete. In the case of  $R$ a one dimensional domain 
there is a complete description of  $Q \otimes_R \hat{R}$ in terms of the completion $\hat{R}$.
\end{abstract}

\maketitle
\begin{center}
\textsl{Dedicated to Hans-Bj\o rn Foxby}
\end{center}
\section{Introduction} 
Let $R$ denote a commutative Noetherian ring. 
For injective $R$-modules $I, J$ it is well-known that 
$\Hom_R(I,J)$ is a flat $R$-module. In order to understand them the first case of interest is when 
$I, J$ are indecomposable (as follows by Matlis' Structure Theory (see e.g. \cite{eM} or \cite{EJ})). 

Let $(R,\mathfrak{m})$ denote a local ring with the injective hull $E = E_R(R/\mathfrak{m})$ 
of the residue field  $k = R/\mathfrak{m}$. In this situation it comes down to understand 
the Matlis dual $\Hom_R(I,E)$ of an injective $R$-module, in particular for $I = E_R(R/\mathfrak{p})$, 
the injective hull of $R/\mathfrak{p}$ for $\mathfrak{p} \in \Spec R$. It was shown (see \cite[3.3.14]{EJ} and 
\cite[3.4.1 (7)]{EJ}) that 
\[
\Hom_R(E_R(R/\mathfrak{p}), E) \simeq \Hom_R(E_R(R/\mathfrak{p}), E_R(R/\mathfrak{p})^{\mu_{\mathfrak{p}}})
\simeq \widehat{R_{\mathfrak{p}}^{\mu_{\mathfrak{p}}}}.
\]
Moreover it follows (see \cite[3.3.10]{EJ}) that 
\[
\mu_{\mathfrak{p}} = \dim_{k(\mathfrak{p})} \Hom_R(k(\mathfrak{p}), E).
\]
Therefore $\Hom_R(E_R(R/\mathfrak{p}), E)$ is the completion of a free $R_{\mathfrak{p}}$-module of rank $\mu_{\mathfrak{p}}$.

Here we shall prove -- as the main result of the paper -- the following result on the Matlis dual 
of a certain $E_R(R/\mathfrak{p})$. 

\begin{theorem} \label{1.1} Let $(R,\mathfrak{m})$ denote a  local ring. Let $\mathfrak{p}$ denote a 
one dimensional prime ideal. Then $\Hom_R(E_R(R/\mathfrak{p}), E) 
\simeq \widehat{R_{\mathfrak{p}}}$ (i.e. it is the completion of a free $R_{\mathfrak{p}}$-module 
of rank one) if and only if $R/\mathfrak{p}$ is complete. 
\end{theorem}

Let $\mathfrak{p}$ denote a one dimensional prime ideal in a local ring $(R,\mathfrak{m})$. 
The equality $\mu_{\mathfrak{p}} = 1$ was proved in \cite{HS} resp. in \cite{wM} in the case of $R$ 
a complete Gorenstein domain resp. in the case of $R$ a complete Cohen-Macaulay domain. The proofs 
are based on the use of the dualizing module of a complete Cohen-Macaulay domain. Note that the 
dualizing module is isomorphic to $R$ in the case of a complete Gorenstein domain. 

Here we use as a basic ingredient Matlis Duality and - as a main step - the reduction to 
the case of $\dim R = 1$ suggested by one of the reviewer's. In the case of a one dimensional domain 
there is a complete description of $\Hom_R(E_R(R),E)$ and $Q \otimes_R \hat{R}$ in terms of the completion $\hat{R}$ (see 
Theorem \ref{3.3} for the precise formulation). 

\section{Proofs} 
In the following $(R,\mathfrak{m})$ denotes always a local ring 
with $E =E_R(R/\mathfrak{m})$ the injective hull of the residue field $R/\mathfrak{m}$. 
Then $D_R(\cdot) = \Hom_R(\cdot,E)$ denotes the Matlis duality functor. 

\begin{remark} \label{2.1} 
(A) Let $X$ be an arbitrary $R$-module.  There is a natural homomorphism 
\[
X \to D(D(X))
\] 
that is always injective. If $(R,\mathfrak{m})$ is complete  it is an isomorphism whenever 
$X$ is an Artinian $R$-module resp. a finitely generated $R$-module (see \cite[p. 528]{eM} 
and \cite[Corollary 4.3]{eM}). Moreover it follows that the map is an 
isomorphism if and only if there is a finitely generated $R$-submodule $Y \subset X$ such that $X/Y$ is an 
Artinian $R$-module. For the proof we refer to \cite{tZ} and also to \cite{BEG} for a generalization.\\
(B) Let $M$ denote a finitely generated $R$-module. Then there is a natural isomorphism 
$M \otimes_R \hat{R} \simeq D(D(M)) $. That is, $M$ is Matlis reflexive if an only if it 
is complete. \\
(C) Let $X$ denote an $R$-module with $\Supp_R X \subset \{\mathfrak{m}\}$. Then $M$ admits the 
structure of an $\hat{R}$-module compatible with its $R$-module structure such 
that $X \otimes_R \hat{R} \to X$ is an isomorphism (see e.g. \cite[(2.1)]{tZ}). Let $M$ denote an $R$-module 
and $N$ an $\hat{R}$ module. Then $\Ext_R^i(M,N), i \in \mathbb{Z},$ has the structure of 
an $\hat{R}$-module. Moreover, here are natural isomorphisms 
\[
\Ext_R^i(M,N) \simeq \Ext_{\hat{R}}^i(M \otimes_R \hat{R},N)
\] 
for all $i \in \mathbb{Z}$ since $\hat{R}$ is a flat $R$-module. 
\end{remark}

As a technical tool we shall need the short exact sequence of the following trivial Lemma.

\begin{lemma} \label{2.2} Let $(R,\mathfrak{m})$ denote a one dimensional domain. Then there 
is a short exact sequence 
\[
0 \to R \to Q \to H^1_{\mathfrak{m}}(R) \to 0
\]
where $Q = \mathbb{Q}(R)$ denotes the quotient field of $R$.
\end{lemma} 

\begin{proof} We start with the  following short exact sequence $0 \to R \to Q \to Q/R \to 0$. 
The long exact local cohomology sequence provides an isomorphism $Q/R \simeq H^1_{\mathfrak{m}}(R)$. 
To this end recall that $H^i_{\mathfrak{m}}(Q) = 0$ for all $i \in \mathbb{Z}$. This proves the statement. 
\end{proof}

As one of the main ingredients of the proof we start with a reduction to the one 
dimensional case suggested by one of the reviewers.

\begin{lemma} \label{3.1} Let $\mathfrak{p}$ be a prime ideal in a local ring 
$(R,\mathfrak{m})$. Let 
\[
\Hom_R(E_R(R/\mathfrak{p}), E) \simeq \widehat{R_{\mathfrak{p}}^{\mu_{\mathfrak{p}}}}
\] 
with $
\mu_{\mathfrak{p}} = \dim_{k(\mathfrak{p})} \Hom_R(k(\mathfrak{p}), E) = 
\dim_{k(\mathfrak{p})}  \Hom_{R/\mathfrak{p}}(k(\mathfrak{p}), E_{R/\mathfrak{p}}(k))$.
\end{lemma} 

\begin{proof} 
Since $k(\mathfrak{p})$ is an $R/\mathfrak{p}$-module the adjunction formula gives the 
following isomorphisms
\[
\Hom_R(k(\mathfrak{p}), E) \simeq \Hom_{R/\mathfrak{p}}(k(\mathfrak{p}), \Hom_R(R/\mathfrak{p},  E)) \simeq 
\Hom_{R/\mathfrak{p}}(k(\mathfrak{p}), E_{R/\mathfrak{p}}(k)).
\]
For the last isomorphism note that $\Hom_R(R/\mathfrak{p}, E) \simeq E_{R/\mathfrak{p}}(k)$.
\end{proof}

Now we are prepared for the main result in the one dimensional case. 

\begin{theorem} \label{3.2} Let $R$ denote a one dimensional local domain and 
$Q = \mathbb{Q}(R)$ its quotient field. There are isomorphisms 
\[
D_{\hat{R}}(D_R(Q)) \simeq Q \otimes_R \hat{R} \simeq Q \oplus \hat{R}/R \text{\; and \;} \hat{R}/R \simeq \Ext_R^1(Q,R).
\]  
Thus $\hat{R}/R$ has a natural structure as a $Q$-vector space, and so it is injective as an $R$-module.
Moreover, $R$ is complete if and only if $\Ext^1_R(Q,R) = 0$, if and only if $D_R(Q) \simeq Q$. 
\end{theorem} 

\begin{proof}  Consider the short exact sequence 
of Lemma \ref{2.1} and apply $\cdot \otimes_R \hat{R}$. It induces a commutative 
diagram with exact rows 
\[
\begin{array}{ccccccccl}
0  & \to & R & \longrightarrow &  Q  & \longrightarrow  & H^1_{\mathfrak{m}}(R)  & \to &  0 \\
    &       &    \downarrow&       & \downarrow  &  & \parallel               &       &        \\
0  & \to & \hat{R} & \to & Q \otimes \hat{R} & \to &H^1_{\mathfrak{m}}(R) &  \to &  0. 
\end{array}
\]
The vertical homomorphism at the right is an isomorphism 
since $H^1_{\mathfrak{m}}(R)$ is an Artinian $R$-module (see Remark \ref{2.1}). Because the vertical homomorphisms are injective 
the snake lemma implies an isomorphism $\hat{R}/R \simeq (Q \otimes_R \hat{R})/Q$. Whence there is the short 
exact sequence 
\[
0 \to Q \to Q \otimes_R \hat{R} \to \hat{R}/R \to 0.
\]
By virtue of Remark \ref{2.1} there is an isomorphism $D_R(Q) \simeq D_{\hat{R}}(Q \otimes_R \hat{R})$. So 
Matlis Duality implies that $D_{\hat{R}}(D_R(Q)) \simeq Q \otimes_R \hat{R}$. 
Since $Q = E_R(R)$ is an injective $R$-module it follows that 
\[
D_{\hat{R}}(D_R(Q)) \simeq Q \otimes_R \hat{R} \simeq Q \oplus \hat{R}/R.
\]
This proves the first isomorphisms. Moreover $\hat{R}/R \simeq \Hom_R(Q,\hat{R}/R)$ since $\hat{R}/R$ admits the structure of a 
$Q$-vector space.

Next  we claim that $\Ext^i_R(Q,\hat{R}) = 0$ for all $i \in \mathbb{Z}$. 
By Matlis Duality and adjointness there are the following isomorphisms 
\[
\Ext^i_R(Q,\hat{R}) \simeq \Ext^i_R(Q,\Hom_R(E,E)) \simeq \Hom_R(\Tor_i^R(Q,E),E).
\] 
So it will be enough to show that $\Tor_i^R(Q,E) = 0$ for all $i \in \mathbb{Z}$. This follows since $Q$ is 
a flat $R$-module and $Q\otimes_R E = 0$. 

With this in mind the long exact cohomology sequence 
of $\Ext^i_R(Q,\cdot)$ applied to the short exact sequence $0 \to R \to \hat{R} \to \hat{R}/R \to 0$ 
induces the isomorphism 
\[ 
\Hom_R(Q,\hat{R}/R) \simeq \Ext_R^1(Q,R).
\] 
This provides the second isomorphism 
of the statement and finishes the proof of the first equivalence. For the second  equivalence note that $\dim_Q \Hom_R(Q,E) = 1 $ 
implies $D_R(D_R(Q)) \simeq Q$ and therefore $\hat{R} =R$ by view of the short exact sequence of Lemma \ref{2.1} and $D_R(D_R(R)) \simeq \hat{R}$. 
\end{proof}

In the following we consider the general case of a one dimensional domain. To this end let 
$\Ass \hat{R} = \{\mathfrak{q}_1,\ldots,\mathfrak{q}_r\}$ denote the set of 
associated prime ideals of the completion $\hat{R}$  of the domain $R$. Then $\mathfrak{q}_i \cap R = (0)$ and 
$\dim \hat{R}/\mathfrak{q}_i = 1$ for $i = 1,\ldots,r$.

\begin{theorem} \label{3.3} Let $(R,\mathfrak{m})$ denote a one dimensional domain. Then there is 
are isomorphisms 
\[
\Hom_R(Q,E) \simeq \oplus_{i=1}^r E_{\hat{R}}(\hat{R}/\mathfrak{q}_i) \text{ and } Q\otimes_R \hat{R} 
\simeq \oplus_{i=1}^r\widehat{\hat{R}_{\mathfrak{q}_i}} ,
\]
where $\widehat{\hat{R}_{\mathfrak{q}_i}}$ denotes the completion of $\hat{R}_{\mathfrak{q}_i}, i = 1,\ldots,r$.
\end{theorem}

\begin{proof} It is known that $\Hom_R(H^1_{\mathfrak{m}}(R), E ) \simeq \Hom_{\hat{R}}(H^1_{\mathfrak{m}\hat{R}}(\hat{R}), E)$ 
is the dualizing module $\omega_{\hat{R}}$ of $\hat{R}$. Its minimal injective resolution as $\hat{R}$-module has the following form 
\[
0 \to \omega_{\hat{R}} \to \oplus_{i=1}^r E_{\hat{R}}(\hat{R}/\mathfrak{q}_i) \to E \to 0
\]
(for these results on the dualizing module see e.g. \cite[Section 3.3]{BH}). By applying the Matlis dual functor $D_R(\cdot)$ to the short exact sequence 
of Lemma \ref{2.2} it provides a short exact sequence of $\hat{R}$-modules 
\[
0 \to \omega_{\hat{R}} \to \Hom_{\hat{R}}(Q \otimes_R \hat{R}, E) \to E \to 0
\]
of $\hat{R}$-modules. Whence there is a commutative diagram with exact rows 
\[
\begin{array}{ccccccccl}
0  & \to & \omega_{\hat{R}} & \to & \Hom_{\hat{R}}(Q \otimes_R \hat{R},E) & \to &E &  \to &  0 \\
    &       &    \parallel &       & \downarrow  &  & \parallel               &       &        \\
0  & \to & \omega_{\hat{R}} & \longrightarrow &  \oplus_{i=1}^r E_{\hat{R}}(\hat{R}/\mathfrak{q}_i)  & \longrightarrow  & E & \to &  0. 
\end{array}
\]
By the snake lemma it yields an isomorphism $\oplus_{i=1}^r E_{\hat{R}}(\hat{R}/\mathfrak{q}_i) \simeq \Hom_{\hat{R}}( Q \otimes_R \hat{R}, E) $ 
and therefore $\Hom_R(Q,E) \simeq \oplus_{i=1}^r E_{\hat{R}}(\hat{R}/\mathfrak{q}_i)$. 
By Matlis Duality (see Remark \ref{2.1}) there are the following isomorphisms 
\[
Q \otimes_R \hat{R} \simeq D_{\hat{R}}(D_{\hat{R}}(Q \otimes_R \hat{R})) \simeq 
\oplus_{i = 1}^r \Hom_{\hat{R}}(E_{\hat{R}}(\hat{R}/\mathfrak{q}_i), E).
\]
By Lemma \ref{3.1} and Theorem \ref{3.2} it provides that $\Hom_{\hat{R}}(E_{\hat{R}}(\hat{R}/\mathfrak{q}_i), E) 
\simeq  \widehat{\hat{R}_{\mathfrak{q}_i}}, i = 1,\ldots,r$. This proves the second statement. 
\end{proof}

\noindent {\it Proof of Theorem \ref{1.1}.} By view of Lemma \ref{3.1} we may reduce the 
computation of $\mu_{\mathfrak{p}}$ to the case of the one dimensional domain $R/\mathfrak{p}$. Then 
the statements are a consequence of Theorem \ref{3.2}. 
\hfill $\Box$

\section{Remarks}
We conclude with a few discussions on the previous results. 

\begin{remark} \label{4.2} Theorem \ref{1.1} does not hold for a prime ideal $\mathfrak{p} \subset R$ in a 
complete local ring $R$ with $\dim R/\mathfrak{p}> 1$. To this end let $(R,\mathfrak{m})$ a complete 
local two dimensional domain. Then $E_R(R) = Q$, the quotient field of $R$.  But now $\Hom_R(Q,E) \simeq Q$ 
can not be true. Assume that it holds. Then the 
natural map $Q \to D(D(Q))$  is an isomorphism too. This can not be the case as follows by view of Remark \ref{2.1}. 

Let $k$ denote a field and $x$ an indeterminate over $k$. Consider the situation of $k[x]_{(x)}$ and 
its completion $k[|x|]$. Then their 
quotient fields are $k(x)$ and $k((x))$ resp. Then 
\[
k(x) \otimes_{k[x]_{(x)}} k[|x|] \simeq k(x) \oplus k[|x|]/k[x]_{(x)}, \text{\; and \;}  \\
\Hom_{k[x]_{(x)}}(k(x),k(x)/k[x]_{(x)}) \simeq k((x)) 
\]
as follows by Theorems \ref{3.2} and \ref{3.3}. 
\end{remark}

\begin{problem} \label{5.3} Let $\mathfrak{p}$ denote a one dimensional prime ideal in a local ring 
$(R, \mathfrak{m})$. Suppose that $R/\mathfrak{p}$ is complete. We know that $\Hom_R(\widehat{R_{\mathfrak{p}}},E)$ 
is again an injective $R$-module. Since the natural homomorphism $E_R(R/\mathfrak{p}) \to I = D(D(E_R(R/\mathfrak{p})))$ is injective 
it turns out that $E_R(R/\mathfrak{p})$ is a direct summand of $I$. By view of Remark \ref{2.1} it can not be 
an isomorphism. 

By the Matlis Structure Theorem 
it follows that $I \simeq \oplus_{\mathfrak{q} \in \Spec R} E_R(R/\mathfrak{q})^{\mu{(\mathfrak{q},I)}}$ where 
\[
\mu(\mathfrak{q},I) = \dim_{k(\mathfrak{q})} \Hom_{R_{\mathfrak{q}}}(k(\mathfrak{q}), I_{\mathfrak{q}})
\]
denotes the multiplicities of the occurrence of $E_R(R/\mathfrak{q})$ in $I$ (see e.g. \cite[Section 3.3]{EJ}). We know that 
$\mu(\mathfrak{q},I) = 0$ for all $\mathfrak{q} \not\subset \mathfrak{p}$ and $\mu(\mathfrak{p},I) \geq 1$. 
It is not clear to us whether $\mu(\mathfrak{p},I)$ is finite or even $1$? Which of the $\mu(\mathfrak{q},I)$ 
are not zero? 
\end{problem} 

\begin{remark} \label{4.3} Let $(R,\mathfrak{m})$ denote a one dimensional domain. One might ask whether 
the $Q$-rank of $Q \otimes_R \hat{R}$ is finite only if $\hat{R} = R$. This is not true as it follows by 
Nagata's Example (E3.3) (see \cite[page 207]{mN}). 
\end{remark}

\section*{Acknowledgements} The author expresses his deep thanks to the reviewers for their help in order to 
simplify, correct and improve his manuscript. One of them suggested the reduction argument 
(see Lemma \ref{3.1}). This allowed an extension and an essential simplification of the author's original arguments. Moreover, the author thanks Winfried Bruns, Ed Enochs and Bill Heinzer for discussions 
around the subjects of the paper.

\end{document}